\documentclass[a4paper,12pt]{amsart}

\usepackage[utf8]{inputenc}
\usepackage[english]{babel}

\usepackage{amsmath,amssymb,amsthm}
\usepackage[alphabetic,nobysame]{amsrefs}

\newtheorem{theorem}{Theorem}[section]
\newtheorem{lemma}[theorem]{Lemma}
\newtheorem{proposition}[theorem]{Proposition}
\newtheorem{cor}[theorem]{Corollary}

\theoremstyle{definition}
\newtheorem{definition}[theorem]{Definition}
\newtheorem{example}[theorem]{Example}
\theoremstyle{remark}
\newtheorem{remark}[theorem]{Remark}

\newcommand{\NN}{\mathbb{N}}
\newcommand{\ZZ}{\mathbb{Z}}
\newcommand{\QQ}{\mathbb{Q}}

\newcommand{\CC}{\mathbb{C}}
\newcommand{\camo}[1]{{\mathcal{C}_{#1}}}
\newcommand{\slgrp}{{\mathrm{SL}}}
\newcommand{\glgrp}{{\mathrm{GL}}}

\newcommand{\rank}{{\mathop{\mathrm{rank}}}}
\newcommand{\id}{{\mathrm{id}}}
\newcommand{\mat}[2]{{\mathrm{Mat}( #1, #2 )}}

\title{The Density Property for Calogero--Moser spaces}
\author{Rafael B. Andrist}
\address{Rafael B. Andrist \\ Department of Mathematics \\
American University of Beirut \\
Beirut, Lebanon}

\date{\today}

\begin{document}
\begin{abstract}
\hfill \\
We prove the algebraic density property for the Calogero--Moser spaces $\camo{n}$, and give a description of the identity component of the group of holomorphic automorphisms of $\camo{n}$.
\end{abstract}

\maketitle

\section{Introduction}

The Calogero--Moser space $\camo{n}$, $n \in \NN$, describes the phase space of a Calogero--Moser system, which is a $n$-particle system in classical physics with a certain Hamiltonian. The idea to study these spaces as quotients of a group action goes back to Kazhdan, Kostant and Sternberg \cite{KKS}. A lot of interesting mathematical properties of these spaces were established by Wilson \cite{wilson}, and since then they have been an ongoing object of study in pure mathematics.

\begin{definition}[\cite{wilson}]
Let $\camo{n}$ be the $n$-th \emph{Calogero--Moser space}, $n \in \NN$, defined as follows:
\[
\begin{split}
\widetilde{\camo{n}} := 
\{
(X,Y) \in \mat{n \times n}{\CC} \times \mat{n \times n}{\CC} \;:\; \\ \rank( [X,Y] + \id ) = 1
\}
\end{split}
\]
Let $\glgrp_n(\CC)$ act on $\widetilde{\camo{n}}$ by simultaneous conjugation in both factors: $G \cdot (X,Y) := (GXG^{-1},GYG^{-1})$. Then
\[
{\camo{n}} := \widetilde{\camo{n}} \,/\!/ \, \glgrp_n(\CC)
\]
\end{definition}

\smallskip

Note that this $\glgrp_n(\CC)$-action also conjugates the commutator $[X,Y]$ and hence leaves the rank-$1$ condition invariant, thus the Calogero--Moser space $\camo{n}$ is well-defined.

One can easily see from the definition that $\camo{1} = \CC^2$. However for $n \geq 2$, the structure of the Calogero--Moser spaces $\camo{n}$ becomes more complicated. In fact, according to \cite{wilson}*{Section~8} the space $\camo{n}$ is diffeomorphic to the Hilbert scheme of $n$ points in $\CC^2$, whose Borel--Moore homology has been calculated by Ellingsrud and Str{\o}mme \cite{homology}.

As shown in \cite{wilson}*{Section~1}, $\camo{n}$ is a smooth irreducible complex affine-algebraic variety of dimension $2n$. Later, Berest and Wilson showed that the group of algebraic automorphisms of $\camo{n}$ acts transitively \cite{transitive}. And more recently, Popov \cite{popov}*{Remark~5} established that $\camo{n}$ is a rational variety. 

\bigskip

We note that the Calogero--Moser space $\camo{n}$ is a flexible variety (see Section \ref{secflex}). This follows directly from the transitive action of the subgroup that is generated by locally nilpotent derivations, which had been established by Berest and Wilson \cite{transitive}, and a more general result characterizing flexibility by Arzhantsev et al.\ \cite{A-Z}. This is implicit also in Berest--Eshmatov--Eshmatov \cite{BEE}*{Conjecture 1, Remark 3}, but seems not to have been stated explicitly in the literature yet. Independent of the more general result, we give a short proof of the flexibility of the Calogero--Moser space $\camo{n}$ using only the transitivity result \cite{transitive}, see Proposition~\ref{propflex}.

\bigskip

The main result of this article is that the Calogero--Moser space $\camo{n}$ enjoys the algebraic density property, see Theorem~\ref{thmdens}. While this is well-known for $\camo{1} = \CC^2$, these results are new for the Calogero--Moser spaces $\camo{n}$, $n \geq 2,$ which provide new examples of flexible manifolds and of manifolds with the algebraic density property. Stein manifolds with the density property enjoy a Runge-type approximation theorem for holomorphic injections, which in turn can be used to show infinite transitivity, see Section \ref{secdens}.

\smallskip

In the Section \ref{secexample} we give an explicit description of $\camo{2}$ as an affine variety, and describe some Calogero--Moser flows and other $\CC^+$-flows of different type on $\camo{2}$.

\section{Flexibility and infinite transitivity}
\label{secflex}

The notion of flexibility was introduced in a paper by Arzhantsev et al.\ \cite{A-Z} and has since stimulated a lot of research in algebraic geometry and complex analysis.

\begin{definition}
Let $G$ be a group acting on a set $M$. We call the action \emph{infinitely transitive} if it is $m$-transitive for every $m \in \NN$.
\end{definition}

\begin{remark}
The term ``infinitely transitive'' may be misleading, but has been established in the literature. However, in the algebraic category, it is clear that $m$-transitivity for every $m \in \NN$ is the best one can hope for. In the holomorphic category, the situation is slighly more sophisticated. By a result of Winkelmann \cite{winkelmann}, every Stein manifold of dimension at least $2$ contains an infinite closed discrete subset whose image under a holomorphic automorphism can't be prescribed. Hence, $m$-transitivity for every $m \in \NN$ is the optimal result also in the holomorphic category.
\end{remark}

\begin{definition}
We call a complex algebraic variety $M$ \emph{flexible} if there exist finitely many locally nilpotent derivations on the ring of regular functions of $M$ such they span the tangent space in each point of $M_{\mathrm{reg}}$.
\end{definition}

For a reference book on locally nilpotent derivations, we refer the reader to the monograph of Freudenburg \cite{LND}. We only want to emphasize the following basic remark as we may switch between $\CC^{+}$-actions and locally nilpotent derivations.

\begin{remark}
The flow of a locally nilpotent derivation is an algebraic $\CC^{+}$-action. Conversely, the time-derivative of an algebraic $\CC^{+}$-action is a locally nilpotent derivation.
\end{remark}

\begin{definition}
The \emph{group of special automorphisms} $\mathrm{SAut}(M)$ of a complex algebraic variety $M$ is the group generated by the flows of locally nilpotent derivations on the ring of regular functions of $M$.
\end{definition}

\begin{theorem}[\cite{A-Z}*{Theorem~0.1}]
\label{thmflextrans}
For a reduced irreducible complex affine-algebraic variety $M$ of dimension $\geq 2$, the following conditions are equivalent.
\begin{enumerate}
\item The group $\mathrm{SAut}(M)$ acts transitively on $M_{\mathrm{reg}}$.
\item The group $\mathrm{SAut}(M)$ acts infinitely transitively on $M_{\mathrm{reg}}$.
\item $M$ is a flexible variety.
\end{enumerate}
\end{theorem}

\begin{example}
Examples of flexible complex manifolds, that are affine, include:
\begin{enumerate}
\item Complex linear groups without non-constant morphisms to $\CC^\ast$
\item Smooth Danielewski surfaces $\{ (x,y,z) \in \CC^3 \,:\,  x \cdot y - p(z) = 0 \}$ where $p$ is a polynomial with simple roots
\end{enumerate}
\end{example}

We give here a short and more elementary proof of the flexibility of the Calogero--Moser spaces that does not rely on the much more general result of Theorem \ref{thmflextrans}  above.

\begin{definition}[\cite{densitycriteria}*{Definition 2.1}]
Let $M$ be a complex algebraic manifold. We call $M$ \emph{tangentially semi-homogeneous} if it is homogeneous (with respect to the algebraic automorphisms of $M$) and admits a \emph{generating set} consisting of one vector, i.e.\ $\exists x_0 \in M, \; \exists v \in T_{x_0} M$ such that the image of $v$ under the induced action of the isotropy group of $x_0$ on $T_{x_0} M$ spans $T_{x_0} M$.
\end{definition}

\begin{proposition} \hfill
\label{propflex}
\begin{enumerate}
\item The Calogero--Moser space $\camo{n}$ is flexible for every $n \geq 1$.
\item The Calogero--Moser space $\camo{n}$ is tangentially semi-homogeneous for every $n \geq 1$.
\end{enumerate}
\end{proposition}

\begin{proof}
Let $\Upsilon = (\Upsilon_1, \Upsilon_2) \colon \camo{n} \to \CC^n/S_n \times \CC^n/S_n$ be the map that sends the matrices $(X,Y)$ to their eigenvalues (counted with multiplicity). The fibres of $\Upsilon_1$, $\Upsilon_2$, respectively, are of dimension $n$, see \cite{wilson}*{Section~1}.

For any point $(X,Y) \in \camo{n}$ with $X$ having $n$ distinct eigenvalues, according to Wilson \cite{wilson}*{Proposition~1.10} we can choose the following representative
\[
\begin{split}
&\left(
\begin{pmatrix}
\lambda_1 & & & & \\
& \lambda_2 & & & \\
& & \ddots & & \\
& & & & \lambda_n \\
\end{pmatrix}, \right. \\
&\quad 
\left.
\begin{pmatrix}
\alpha_1 & (\lambda_1 - \lambda_2)^{-1} & \dots &   (\lambda_1 - \lambda_n)^{-1} \\
(\lambda_2 - \lambda_1)^{-1} & \alpha_2 & \ddots & \vdots \\
\vdots & \ddots & \ddots & (\lambda_{n-1} - \lambda_n)^{-1} \\
(\lambda_n - \lambda_1)^{-1} & \dots & (\lambda_{n} - \lambda_{n-1})^{-1} & \alpha_n \\
\end{pmatrix}
\right)
\end{split}
\]
where $(\alpha_1, \dots, \alpha_n) \in \CC^n$ is arbitrary.
This parametrization provides a trivialization of the tangent bundle over the subset $\{(X,Y) \in \camo{n} \,:\, X \text{ has } n \text{ distinct eigenvalues} \}$ which is biholomorphic to
\[
\Omega := \left( \CC^n \setminus \mathrm{diag} \right) \times \CC^n \ni (\lambda_1, \dots, \lambda_n, \alpha_1, \dots \alpha_n)
.\]
Note that the complex numbers $\alpha_1, \dots, \alpha_n$ are \textbf{not} the eigenvalues of $Y$.

\smallskip

The flow maps $(X, Y) \mapsto (X, Y + t X^k)$, where $k \in \NN_0$, are called \emph{Calogero--Moser flows}. It is easy to see that they are well-defined on $\camo{n}$, since they leave the commutator $[X,Y]$ invariant. Since these algebraic flow maps exist for all $t \in \CC$, they correspond to complete algebraic vector fields, which we denote by $X^k \frac{\partial}{\partial Y}$, slightly abusing notation.
We will also call their ``dual'' counterparts $(X, Y) \mapsto (X+ t Y^k, Y)$ \emph{Calogero--Moser flows}, whose corresponding vector fields we denote by $Y^k \frac{\partial}{\partial X}$.

\begin{enumerate}
\item Choose a point $(X_0, Y_0) \in \Omega$ such that $\lambda_1, \dots, \lambda_n, \alpha_1, \dots, \alpha_n$ do not satisfy any non-trivial polynomial relation over $\QQ$. Then the diagonal entries  of $X_0, X_0^2, \dots, X_0^n$, i.e.\ $(\lambda_1^k, \dots, \lambda_n^k)$ for $k = 1, \dots, n$, will span a $n$-dimensional vector space. Similarly, the diagonal entries of $Y_0, Y_0^2, \dots, Y_0^n$, which are rational functions of $\lambda_1, \dots, \lambda_n, \alpha_1, \dots, \alpha_n$ with integer coefficients, will span a $n$-dimensional vector space. Then the vector fields $X^k \frac{\partial}{\partial Y}$ for $k = 1, \dots, n$ span the tangent space in $(X_0, Y_0)$ along the fibre of $\Upsilon_2$. Let $\mu_1, \dots, \mu_n$ denote the eigenvalues of $Y_0$, which then can't satisfy any non-trivial polynomial relation over $\QQ$ either.
Using the same coordinates as above, but with roles of $X$ and $Y$ interchanged, we see that the vector fields $Y^k \frac{\partial}{\partial X}$ for $k = 1, \dots, n$ span the tangent space in $(X_0, Y_0)$ along the fibre of $\Upsilon_1$.
Together, these $2n$ vector fields with complete algebraic flows span the tangent space in $(X_0, Y_0)$ and thus on a Zariski-open subset of $\camo{n}$.

Using the well-known transitivity \cite{transitive}*{Theorem~1.3} of the algebraic automorphism group of the space, the flexibility follows from standard arguments: Let $A \subset \camo{n}$ be the subvariety where the vector fields above do not span. Pick a point $x_0 \in A$ and a point $x_1 \in \camo{n} \setminus A$. By transitivity, there exists an algebraic automorphism $\Phi$ of $\camo{n}$ with $\Phi(x_0) = x_1$. Conjugate the vector fields that span in $x_1$ with $\Phi$, to obtain spanning vector fields in $x_0$. This process will terminate after finitely many steps, since it reduces the dimension of the connected component of $A$ that contains $x_1$ in every step, and since there are only finitely many connected components.
\item
Let $(X_0,Y_0) \in \camo{n}$ be as above with the same coordinate neighborhood. Flexibility implies straight-forward homogeneity. To prove tangential semi-homogeneity, we need to find a tangent vector that is a generating set.
For $k \in \NN_0$, the functions
\begin{align*}
F_k(X, Y) &= (X + ( \mathop{\mathrm{tr}} Y - \mathop{\mathrm{tr}} Y_0 ) \cdot Y^k, Y) \\
G_k(X, Y) &= (X , Y + ( \mathop{\mathrm{tr}} X - \mathop{\mathrm{tr}} X_0 ) \cdot X^k)
\end{align*}
are well-defined and fix the point $(X_0, Y_0) \in \camo{n}$.
For their derivatives we obtain
\begin{align*}
d_{(X_0,Y_0)} F_k &= \begin{pmatrix}
 \id & d_{Y_0} \mathop{\mathrm{tr}} Y \cdot Y_0^k \\
   0 & \id 
\end{pmatrix} \\
d_{(X_0,Y_0)} G_k &= \begin{pmatrix}
 \id & 0 \\
   d_{X_0} \mathop{\mathrm{tr}} X \cdot X_0^k & \id 
\end{pmatrix}
&= \begin{pmatrix}
 \id & 0 \\
  \mathop{\mathrm{diag}}(\lambda_1^k, \dots, \lambda_n^k) & \id 
\end{pmatrix}
\end{align*}
Any vector $w \in T_{(X_0,Y_0)}\camo{n}$ whose projections to each of the one-dimensional subspaces spanned by $\frac{\partial}{\partial \lambda_k}, \frac{\partial}{\partial \alpha_k}$, $k = 1, \dots, n$ is non-trivial, will be a generating set. \qedhere
\end{enumerate}
\end{proof}

\begin{cor}
The algebraic automorphism group of $\camo{n}$ acts infinitely transitively on $\camo{n}$.
\end{cor}

\begin{remark}
This corollary does not make use of the two-transitivity results of Berest--Eshmatov--Eshmatov \cite{BEE}, but only of the flexibility (and hence implicitly of the transitivity result). This corollary is also implied by a recently established result of Kuyumzhiyan \cite{kuyu} where they prove that a smaller group, which is a proper subgroup of the algebraic automorphism group, already acts infinitely transitively on $\camo{n}$.
\end{remark}

\section{Density property}
\label{secdens}

The density property was introduced by Varolin \cites{Varolin1,Varolin2} in order to describe that the holomorphic automorphism group of a complex manifold is large. The most important consequence is the so-called Anders\'en--Lempert Theorem. Among many interesting geometric consequences it implies that for a Stein manifold with the density property, the group of holomorphic automorphisms acts infinitely transitively.

\begin{definition}[\cite{Varolin1}]
Let $M$ be Stein manifold. We say that $M$ enjoys the \emph{density property} if the Lie algebra that is generated by all complete holomorphic vector fields on $M$ is dense (w.r.t.\ compact-open topology) in the Lie algebra of all holomorphic vector fields on $M$.
\end{definition}

In practice it is often simpler to fulfil the following, more algebraic condition, which then implies the density property:

\begin{definition}[\cite{Varolin1}]
Let $M$ be complex affine-algebraic manifold. We say that $M$ enjoys the \emph{algebraic density property} if the Lie algebra that is generated by all complete polynomial vector fields on $M$ is agrees with the Lie algebra of all polynomial vector fields on $M$.
\end{definition}

There is a close relation between flexibility and density property, in particular due to the fact that locally nilpotent derivations are complete polynomial vector fields. So far, all known examples of complex affine-algebraic manifolds that are flexible, also enjoy the algebraic density property. However, not all manifolds with the algebraic density property are flexible, a counter\-example is $\CC \times \CC^\ast$.

\begin{theorem}[Anders\'en--Lempert \cite{AL2}, Forstneri\v{c}--Rosay \cites{FR,FR-err}, Varolin \cite{Varolin1}]
\label{andlemp}
Let $M$ be a Stein manifold with the density property and with a distance function $d$. Let $\Omega \subseteq M$ be a Stein open subset and $\varphi \colon [0,1] \times \Omega \to M$ be a continuously differentiable map such that
\begin{enumerate}
\item $\varphi_0 \colon \Omega \to M$ is the natural embedding,
\item $\varphi_t \colon \Omega \to M$ is holomorphic and injective for every $t \in [0,1]$,
\item {$\varphi_t(\Omega)$ is a Runge subset of $M$ for every $t \in [0,1]$.}
\end{enumerate}
Then for every $\varepsilon > 0$ and for every compact $K \subset \Omega$ there exists a continuous family $\Phi \colon [0, 1] \to \mathrm{Aut}(M)$
with $\Phi_0 = \id_M$ and \[
\sup_{x \in K} d( \varphi_t(x), \Phi_t(x) ) < \varepsilon \] for every $t \in [0,1]$.

\smallskip
Moreover, these automorphisms can be chosen to be compositions of flows of completely integrable generators of any Lie subalgebra of holomorphic vector fields that is dense in the Lie algebra of all holomorphic vector fields on $M$.
\end{theorem}

\begin{cor}[\cite{Varolin1}]
Let $M$ be a Stein manifold with the density property. 
Then the natural action of the group of holomorphic automorphisms of $M$ on $M$ is infinitely transitive.
\end{cor}

\bigskip

A crucial contribution for proving the algebraic density property are the following definition and theorem of Kaliman and Kutzschebauch:

\begin{definition}[\cite{densitycriteria}*{Definition~2.5}]
\label{defcompat}
Let $\Theta_1$ and $\Theta_2$ be non-trivial polynomial vector fields on an complex affine-algebraic manifold $M$ such that $\Theta_1$ is a locally nilpotent derivation on $\CC[M]$, and $\Theta_2$ is either also locally nilpotent or semi-simple. That is, $\Theta_j$ generates an algebraic action of $H_j$ on $M$
where $H_1 \cong \CC^+$ and $H_2$ is isomorphic to either $\CC^+$ or $\CC^\ast$. We say that $\Theta_1$ and $\Theta_2$ form a \emph{compatible pair} if (i) the vector space $\mathrm{span}(\ker \Theta_1 \cdot \ker \Theta_2)$ contains a nonzero ideal in $\CC[M]$ and (ii) some element $a \in \ker \Theta_2$ is of degree $1$ with
respect to $\Theta_1$, i.e.\ $\Theta_1(a) \in \ker \Theta_1 \setminus \{0\}$.
\end{definition}

\begin{theorem}[\cite{densitycriteria}*{Theorem~2, and Remark~2.7}]
\label{thmcompatpair}
Let $M$ be a complex-affine algebraic manifold that
\begin{enumerate}
\item is tangentially semi-homogenous, and
\item admits a compatible pair,
\end{enumerate}
then $M$ enjoys the algebraic density property.
\end{theorem}

\begin{example} Known examples of complex manifolds with the algebraic density property include the following.
\begin{enumerate}
\item Homogeneous spaces of linear algebraic groups $G/H$ that are affine and whose connected components are different from $\CC$ or $(\CC^\ast)^n$, see \cite{MR3623226}.
\item Smooth Danielewski surfaces $\{ (x,y,z) \in \CC^3 \,:\,  x \cdot y - p(z) = 0 \}$ where $p$ is a polynomial with simple roots, see \cite{MR2350038}.
\item A family of cubics including the Koras--Russel cubic threefold, see \cite{MR3513546}.
\item Gizatullin surfaces with a reduced degenerate fibre and a transitive action of the algebraic automorphism group, see \cite{MR3833804}.
\end{enumerate}
\end{example}

\begin{theorem}
\label{thmdens}
The Calogero--Moser spaces $\camo{n}$ have the algebraic density property.
\end{theorem}
\begin{proof}
Since it has already been established that $\camo{n}$ is tangentially semi-homogeneous in Proposition \ref{propflex}, we only need to find a compatible pair in order to apply Theorem \ref{thmcompatpair} above.

\smallskip

We consider two complete $\CC^+$-actions on $\camo{n}$ with times $t, s$, respectively:
\[
(X,Y) \mapsto (X + t \cdot Y,Y), \qquad (X,Y) \mapsto (X,Y + s \cdot X)
\]
These actions are non-trivial on $\camo{n}$ since they change the eigenvalues and therefore the conjugacy classes of the representatives $(X,Y)$.
In order to show that they form a compatible pair, we embed those two actions into a $\slgrp_2(\CC)$-action\footnote{Unfortunately, the quotient of an affine-algebraic variety by $\CC^+$-action does not need to be affine. However this is always true for quotients of affine varieties by reductive subgroups where moreover geometric and categorical quotients agree. This then also holds for quotients of $\CC^+$, as long as this $\CC^+$ is a subgroup of a reductive group that is acting. See \cite{LND} for details.} on $\mat{n \times n}{\CC} \times \mat{n \times n}{\CC}$:
\[
\begin{pmatrix}
a_{11} & a_{12} \\
a_{21} & a_{22}
\end{pmatrix}
\cdot (X, Y) = ( a_{11} X + a_{12} Y, a_{21} X + a_{22} Y )
\]
Since $[a_{11} X + a_{12} Y, a_{21} X + a_{22} Y] = (a_{11} a_{22} - a_{12} a_{21}) \cdot [X, Y] = [X,Y]$ and $a_{11} G X G^{-1} + a_{12} G Y G^{-1} = G (a_{11} X + a_{12} Y ) G^{-1}$ as well as $a_{21} G X G^{-1} + a_{22} G Y G^{-1} = G (a_{21} X + a_{22} Y ) G^{-1}$, this action descends to $\camo{n}$. The $\CC^+$-actions in question are then embedded as subgroups 
\begin{align*}
&\begin{pmatrix}
1 & t \\ 0 & 1
\end{pmatrix}, \; t \in \CC \\
&\begin{pmatrix}
1 & 0 \\ s & 1
\end{pmatrix}, \; s \in \CC, \quad \text{respectively.}
\end{align*}
Thus, the vector fields corresponding to these two $\CC^+$-actions above form a compatible pair according to \cite{densitycriteria}*{Lemma~3.6}.
\end{proof}

\begin{remark}
We believe that the following vector fields also form a compatible pair on $\camo{n}$:
\[
(X,Y) \mapsto (X + t \cdot \id,Y), \qquad (X,Y) \mapsto (X,Y + s \cdot \id)
\]
Since they are not embedded in a $\slgrp_2(\CC)$ action, more work would be needed to show that their kernels are big enough. In Section \ref{secexample}, we show that these two vector fields indeed form a compatible pair for $n = 2$.
\end{remark}

As a consequence of Theorem \ref{andlemp}, we obtain a description of the holomorphic automorphisms.

\begin{definition}
Let $M$ be a complex manifold and let $\Theta$ be a complete holomorphic vector field on $M$ with flow map $\varphi_t$. Let $f \colon M \to \CC$ be a holomorphic function.
\begin{enumerate}
\item If $f \in \ker \Theta$, then we call the vector field $f \cdot \Theta$ a \emph{shear (vector field)} of $\Theta$, and we call the corresponding flow map a \emph{shear} of $\varphi_t$.
\item If $f \in \ker (\Theta \circ \Theta)$, but  $f \notin \ker \Theta$, then we call the vector field $f \cdot \Theta$ an \emph{overshear (vector field)} of $\Theta$, and we call the corresponding flow map an \emph{overshear} of $\varphi_t$.
\end{enumerate}
\end{definition}

\begin{remark}
For a complete vector field $\Theta$, the shear or overshear vector field $f \cdot \Theta$ is complete as well \cite{shear}. If $\Theta$ is a locally nilpotent derivation and $f \in \ker \Theta$ a regular function, then the shear $f \cdot \Theta$ is a locally nilpotent derivation as well, and sometimes called a \emph{replica} of $\Theta$ \cite{A-Z}*{Example~1.19}. The flow maps can be given explicitly \cite{fibred}*{Lemma~3.3}.
\begin{enumerate}
\item If $\Theta(f) = 0$, then the flow is given by $\varphi_{\displaystyle t \cdot f(x)}(x)$.
\item If $\Theta(\Theta(f)) = 0$, then the flow is given by $\varphi_{\displaystyle \varepsilon(t \Theta_x f) \cdot t \cdot f(x)}(x)$ where $\varepsilon \colon \CC \to \CC$ is given by $\varepsilon(\zeta) = \sum_{k=1}^\infty \frac{\zeta^{k-1}}{k!} = \frac{e^\zeta - 1}{\zeta}$
\end{enumerate}
\end{remark}

\begin{theorem}
\label{camoautos}
Let $\varphi$ be a holomorphic automorphism of $\camo{n}$ that is connected to the identity by a continuously (real) differentiable path of holomorphic automorphisms of $\camo{n}$. Then $\varphi$ can be approximated, uniformly on compacts of $\camo{n}$, by compositions of automorphisms of the form $(X, Y) \mapsto (X, Y + p(X))$ and $(X, Y) \mapsto (X + q(Y),Y)$, and their shears and overshears. Here, $p$ and $q$ are polynomials in one complex variable, evaluated on matrices.
\end{theorem}

\begin{example}
Consider the Calogero--Moser flow
\[
\varphi_t((X,Y)) = (X, Y + t \cdot \id)
\]
Obviously all well-defined polynomial functions of $X$ only are invariant functions, in particular all the elementary symmetric functions of the eigenvalues of $X$ and all polynomials therein. 
A shear of $\varphi_t$ is for example given by 
\[
(X,Y) \mapsto (X, Y + t \cdot (\mathop{\mathrm{tr}}(X) + \det(X)) \cdot \id)
\]
\end{example}

\begin{proof}[Proof of Theorem \ref{camoautos}]
This follows directly from Theorem \ref{andlemp} and the fact that the proof of transitivity in \cite{transitive}*{Theorem 1.3} requires only Calogero--Moser flows, and hence also flexibility requires only Calogero--Moser flows and their shears. The proof of the density property also involves shears and overshears of such Calogero--Moser flows.
\end{proof}

\begin{remark}
It is not obvious whether the automorphism given by \[(X,Y) \mapsto (Y^\intercal,X^\intercal)\] is connected to the identity by a continuously differentiable path within the holomorphic automorphisms of $\camo{n}$ for $n \geq 2$. 
\end{remark}

\section{Example: automorphisms of $\camo{2}$}
\label{secexample}

In the following we will study $\camo{2}$. Our goal is to give an explicit description of the manifold (up to a finite cover) as well as some of the complete vector fields and their kernels.
 
We may choose $Y = \begin{pmatrix} \lambda & 1 \\ 0 & \lambda+\varepsilon \end{pmatrix}$ as a representative which will work both for the case of different eigenvalues and the double eigenvalue case which necessarily can't be diagonal due to the rank-$1$ condition. Note that for $\varepsilon \neq 0$, there are two representatives for the same element:
\[
Y = \begin{pmatrix} \lambda & 1 \\ 0 & \lambda+\varepsilon \end{pmatrix} \sim  \begin{pmatrix} \lambda+\varepsilon & 1 \\ 0 & \lambda \end{pmatrix} 
\]
The rank-$1$ condition now becomes
\begin{equation}
\label{eqrank}
\tag{R}
  \varepsilon^2 x_{21} x_{12}
- \varepsilon x_{21} ( x_{22} - x_{11} )
+ 1 - x_{21}^2 = 0
\end{equation}
which forces $x_{21} \neq 0$.

The subgroup of $\glgrp_2(\CC)$ that leaves the given $Y$ invariant, can be parameterized by
\[G = \begin{pmatrix}g & h\\
0 & \epsilon h+g\end{pmatrix}\]
and we have to require $g \neq 0$. Then $h \in \CC$, if $\varepsilon = 0$, and $h \in \CC \setminus \{-g/\varepsilon\}$ otherwise.

Our goal is to find a canonical representative for $X$ with $x_{12} = 0$. To this end, we choose $g = x_{21}$ and note that
\[
G \cdot X \cdot G^{-1} = 
\begin{pmatrix}{x_{11}}+h & (h {x_{22}}+{x_{12}} {x_{21}}-h {x_{11}}-{{h}^{2}})/({x_{21}}+\varepsilon  h) \\
{x_{21}}+\varepsilon  h & {x_{22}}-h\end{pmatrix}
\]
We can always solve $h x_{22}+x_{12} x_{21} - h x_{11}-h^2 = 0$ for $h$:
\[
2 h = \sqrt{(x_{22} - x_{11})^2 + 4 x_{12} x_{21}} + (x_{22} - x_{11})
\]
where we can take any complex root. We just need to check carefully if there exist different representatives for the same element. Using the formula above, we can see that there are generally two representatives which can be mapped to each other using $h = x_{22} - x_{11}$, and only one representative if $x_{22} = x_{11}$, i.e.\ when $X$ has a double eigenvalue:
\[
X = \begin{pmatrix}
x_{11} & 0 \\ x_{21} & x_{22}
\end{pmatrix} \sim \begin{pmatrix}{x_{22}} & 0\\
\epsilon  \left( {x_{22}}-{x_{11}}\right) +{x_{21}} & {x_{11}}\end{pmatrix}
\]

Setting $\delta = x_{22} - x_{11}$, we obtain
\[
\widehat{\camo{2}} := \{ (\lambda, \varepsilon, x_{11}, x_{21}, \delta) \in \CC^5 \; : \; x_{21}^2 + \varepsilon \cdot \delta \cdot x_{21} - 1 = 0 \}
\]

We can of course see easily that $\widehat{\camo{2}}$ is a smooth affine variety of dimension $4$. Moreover, we can find compatible pairs of $\CC^+$-actions and prove the algebraic density property for $\widehat{\camo{2}}$. However, we need to pass from $\widehat{\camo{2}}$ to $\camo{2}$ by identifying the pairs of equivalent representatives by two $\mathbb{Z}_2$-actions:
\begin{align}
\label{eqz2one}
(\lambda, \varepsilon, x_{11}, x_{21}, \delta) &\mapsto (\lambda, \varepsilon, x_{11} + \delta, x_{21} + \varepsilon \delta, -\delta) \\
\label{eqz2two}
(\lambda, \varepsilon, x_{11}, x_{21}, \delta) &\mapsto (\lambda+\varepsilon, -\varepsilon, x_{11}, x_{21} + \varepsilon \delta, \delta)
\end{align}
Where the first action as obtained by the proper choice of $G$ above ($h = \delta$), interchanging $x_{11}$ and $x_{22}$ without affecting the representative $Y$, and where the second action is obtained by conjugation with $\begin{pmatrix}
1 & 0 \\ -\varepsilon & 1
\end{pmatrix}$ interchanging the eigenvalues of $Y$ and which is also affecting $X$.

\begin{lemma}
The ring of regular functions of $\camo{2}$ is generated by 
\[
\begin{split}
\delta^2 = (\mathop{\mathrm{tr}} X)^2 - 4 \det X, \; \varepsilon^2 = (\mathop{\mathrm{tr}} Y)^2 - 4 \det Y, \\ 2 \lambda + \varepsilon = \mathop{\mathrm{tr}} Y, \; 2 x_{11} + \delta = \mathop{\mathrm{tr}} X, \; x_{21} + \frac{1}{x_{21}}
\end{split}
\]
\end{lemma}
\begin{proof}
We have $\camo{2} = \widehat{\camo{2}} \,/\!/\, \ZZ_2 \times \ZZ_2$. The generators can be obtained by symmetrizing the generators of the ring of regular functions of $\widehat{\camo{2}}$.
\end{proof}

We have the following obvious $\CC^+$-actions on $\widehat{\camo{2}}$ that commute with both $\ZZ_2$-actions \eqref{eqz2one} and \eqref{eqz2two} above, and hence give rise to $\CC^+$-actions on $\camo{2}$:
\begin{align*}
(\lambda, \varepsilon, x_{11}, x_{21}, \delta) &\stackrel{\varphi_t}{\mapsto} (\lambda, \varepsilon, x_{11} + t, x_{21}, \delta) \\
(\lambda, \varepsilon, x_{11}, x_{21}, \delta) &\stackrel{\psi_s}{\mapsto} (\lambda + s, \varepsilon, x_{11}, x_{21}, \delta)
\end{align*}

Moreover, these two actions, corresponding to $(X + t \cdot \id, Y)$ and $(X, Y + s \cdot \id)$ respectively, commute. The following functions (and consequently all polynomials therein) on $\camo{2}$ are invariant under $\varphi_t$ and $\psi_t$:
\[
\delta^2, \varepsilon^2, x_{21} + \frac{1}{x_{21}} 
\]
Note that the rank-$1$ condition \eqref{eqrank} gives a non-trivial polynomial relation between these three functions. We can now also see that for example
\[
(X,Y) \mapsto (X, Y + \varepsilon^2 t  \cdot \id) = (X, Y + ((\mathop{\mathrm{tr}} Y)^2 - 4 \det Y) \cdot t \cdot \id) 
\]
is a $\CC^+$-action which is not a Calogero--Moser flow.

We have now explicitly identified two vector fields that form a compatible pair on $\camo{2}$ according to Definition \ref{defcompat}:
\begin{align*}
\Theta_1 &= \frac{d}{d s} \psi_s = \frac{\partial}{\partial \lambda} \\
\Theta_2 &= \frac{d}{d t} \varphi_t = \frac{\partial}{\partial x_{11}} \\
a &= 2 \lambda + \varepsilon
\end{align*}
with $\Theta_2(a) = 0$, $\Theta_1(a) \neq 0$ and $\Theta_1(\Theta_1(a)) = 0$.
For the kernels we obtain that
\begin{align*}
\ker \Theta_1 &\supset \mathop{\mathrm{span}} \{ \delta^2, \varepsilon^2, x_{21} + \frac{1}{x_{21}}, 2 x_{11} + \delta \} \\
\ker \Theta_2 &\supset \mathop{\mathrm{span}} \{ \delta^2, \varepsilon^2, x_{21} + \frac{1}{x_{21}}, 2 \lambda + \varepsilon \}
\end{align*}

Note that the kernel of a derivation is multiplicatively closed. It is then easy to see that $\mathop{\mathrm{span}} (\ker \Theta_1 \cdot \ker \Theta_2)$ contains the ideal generated by $(2 x_{11} + \delta) \cdot (2 \lambda + \varepsilon)$. Hence $\Theta_1$ and $\Theta_2$ form a compatible pair on $\camo{2}$.

We conclude with the remark that the automorphism $(X,Y) \mapsto (Y^\intercal,X^\intercal)$ is given by:
\[
(\lambda, \varepsilon, x_{11}, x_{21}, \delta) \mapsto (x_{11}, \delta, \lambda, x_{21}, \varepsilon)
\]
It it is not clear whether this automorphism is in the connected component of the identity.


%


\end{document}